\documentclass[12pt,makeidx]{amsart}

\usepackage{amssymb,mathrsfs, amsmath, amsthm}
\usepackage{url,titletoc,enumitem}

\usepackage{nag}

\usepackage[usenames,dvipsnames]{xcolor}
\definecolor{darkblue}{rgb}{0.0, 0.0, 0.55}

\newtheorem{theorem}{Theorem}[section]
\newtheorem{cor}[theorem]{Corollary}
\newtheorem{lemma}[theorem]{Lemma}

\newtheorem{prop}[theorem]{Proposition}

\def\C{\mathbb {C}}

\def\cS{\mathcal S}

\newcommand{\df}[1]{{\bf{#1}}{\index{#1}}}

\makeindex

\usepackage{upgreek}
\usepackage{cmap}

\title{Compact Sets in the Free Topology}

\author[M. Augat]{Meric Augat}
\address{Meric Augat, Department of Mathematics\\
  University of Florida\\ Gainesville
   }
   \email{mlaugat@ufl.edu}

\author[S. Balasubramanian]{Sriram Balasubramanian${}^1$}
\address{Sriram Balasubramanian, Department of Mathematics \\ IIT Madras \\ Chennai - 600036, India}
\email{bsriram@iitm.ac.in}
\thanks{${}^1$Supported by the New Faculty Initiative Grant (MAT/15-16/836/NFIG/SRIM) of IIT Madras.}

\author[S. McCullough]{Scott McCullough${}^2$}
\address{Scott McCullough, Department of Mathematics\\
  University of Florida\\ Gainesville 
   }
   \email{sam@math.ufl.edu}
\thanks{${}^2$Research supported by the NSF grant DMS-1361501}


\begin{document}

\begin{abstract}
Subsets of the set of $g$-tuples of matrices that are closed with respect to direct sums and compact in the free topology are characterized.  They are, in a dilation theoretic sense, contained in the hull of a single point. 
\end{abstract}

\maketitle

\section{Introduction}
Given positive integers $n,g$, let $M_n(\C)^g$ denote the set of $g$-tuples of $n\times n$ matrices.
Let $M(\C)^g$ denote the sequence $(M_n(\C)^g)_n$. A subset $E$ of $M(\C)^g$ is a sequence $(E(n))$ where $E(n)\subset M_n(\C)^g.$    The free topology \cite{AM14} has as a basis free sets of the form  $G_\delta = (G_\delta(n))$, where 
\[
 G_\delta(n) =\{X\in M_n(\C)^g: \| \delta(X) \| <1\},
\]
and $\delta$ is a (matrix-valued) free polynomial.  Agler and McCarthy \cite{AM14} prove the remarkable result that a bounded  free 
function on a basis set $G_\delta$ is uniformly approximable by polynomials on each smaller set of the form
\[
 K_{s\delta}=\{X\in M(\C)^g: \|\delta(X)\|\le s\}, \ \ 0\le s<1.
\]
For the definitive treatment of free function theory, see \cite{KVV}. 

 Sets $E\subset M(\C)^g$ naturally arising in free analysis (\cite{AM15, BMV, BKP16, HKN14, KV, KS, Pas, Voi10} is a sampling of the references) are typically closed with respect to direct sums  in the sense that if $X\in E(n)$ and $Y\in E(m)$, then
\[
  X\oplus Y = \left ( \begin{pmatrix} X_1 & 0\\ 0 & Y_1 \end{pmatrix}, \ldots, \begin{pmatrix} X_g & 0 \\ 0 & Y_g\end{pmatrix} \right ) \in E(n+m).
\]

Theorem \ref{thm:freecompact} below, characterizing free topology compact sets that are closed with respect to directs sums,  is the main result of this article. A tuple $Y\in M_n(\C)^g$ \df{polynomially dilates} to a tuple $X\in M_N(\C)^g$ if there is an isometry $V:\C^n\to \C^N$ such that for all free polynomials $p$,
\[
 p(Y)=V^* p(X)V.
\]
An \df{ampliation} of $X$ is a tuple of the form $I_k\otimes X$, for some positive integer $k$.
The \df{dilation hull} of $X\in M(\C)^g$ is the set of all $Y \in M(\C)^g$ that dilate to an ampliation of $X$.

\begin{theorem}
 \label{thm:freecompact}
 A subset $E$ of $M(\C)^g$ that is closed with respect to direct sums  is compact if and only if it is contained in the polynomial dilation hull of an $X\in E.$ 
\end{theorem}

\begin{cor}
 If $E\subset M(\C)^g$ is closed with respect to direct sums and is compact in the free topology, then there exists a non-zero free polynomial $p$ such that $E$ is a subset of the zero set of $p$; i.e., $p(Y)=0$ for all $Y\in E$. In particular, there is an $N$ such that for $n\ge N$ the set $E(n)$  has empty interior.
\end{cor}

\begin{proof}
 By Theorem \ref{thm:freecompact}, there is an $n$ and $X\in E(n)$ such that each $Y\in E$ polynomially dilates to an ampliation of $X$. Choose a nonzero scalar free polynomial $p$ such that $p(X)=0$ (using the fact that the span of $\{w(X): \mbox{$w$ is a word}\}$ is a subset of the finite dimensional vector space $M_n(\C)$). It follows that $p(Y)=0$ for all $Y$. Hence $E$ is a subset of the zero set of $p$. It is well known (see for instance the Amistur-Levitzki Theorem \cite{Row}) that the zero set  $p$ in $M_n(\C)^g$  must have empty interior for sufficiently large $n$. 
\end{proof}

The authors thank Igor Klep for a fruitful correspondence which led to this article. 
The proof of Theorem \ref{thm:freecompact} occupies the remainder of this article.

\section{The proof of Theorem~\ref{thm:freecompact}}

\begin{prop}
\label{prop:igor}
  Suppose $E\subset M(\C)^g$ is nonempty and  closed with respect to direct sums.  If for each $X\in E$ there is a matrix-valued free polynomial $\delta$ and a $Y\in E$ such that
\[
 \|\delta(X)\|< \|\delta(Y)\|,
\]
 then $E$ is not compact in the free topology.
\end{prop}

\begin{proof}
By hypothesis, for each $X \in E$  there is a matrix-valued polynomial $\delta_X$ and {$Y_X \in E$} such that $\|\delta_X(X)\|< 1<\|\delta_X(Y_X)\|$.  The collection $\mathcal G=\{G_{\delta_X}: X\in E\}$ is an open cover of $E.$  Suppose $S\subset E$ is a finite.  Observe that for each $X\in S$, $Y_X \in E \setminus G_{{\delta_X}}.$ Since $E$ is closed with respect to direct sums, $Z=\oplus_{X\in S}  Y_X \in E$. On the other hand, for a fixed $W\in S$,
\[
 \|\delta_W(Z)\| \ge \| \delta_W(Y_W)\| > 1. 
\]
 Thus $Z\notin G_{\delta_W}$ and therefore $Z\in E$  but $Z\notin \cup_{X\in S} G_{{\delta_X}}.$ 
Thus $\mathcal G$ admits no finite subcover of $E$ and therefore $E$ is not compact.
\end{proof}

The following lemma is a standard result. 

\begin{lemma}
\label{lem:dilation}
 Suppose $X,Y\in M(\C)^g$. The tuple $Y$ polynomially dilates to an ampliation of $X$ if and only  if
\[
 \|\delta(Y)\|\le \|\delta(X)\|
\]
 for every free matrix-valued polynomial $\delta.$
\end{lemma}

\begin{proof}
Let $\mathcal P$ denote the set of scalar free polynomials in $g$ variables. Given a tuple $Z\in M_n(\C)^g$, let $\cS(Z)=\{p(Z): p\in\mathcal P\}\subset M_n(\C)$. The set $\cS(Z)$ is a unital operator algebra.   Let $m$ and $n$ denote the sizes of $Y$ and $X$ respectively. The hypotheses thus imply that the unital homomorphism $\lambda:\cS(X)\to \cS(Y)$ given by $\lambda(p(X)) = p(Y)$ is well defined and completely contractive. Thus by Corollary 7.6 of \cite{P}, it follows that there exists a completely positive map $\varphi:M_n(\C)\to M_m(\C)$  extending  $\lambda$. 
By Choi's Theorem \cite{P}, there exists an $M$ and, for $1\le j\le M$, mappings $W_j:\C^{m}\to \C^n$ such that $\sum W_j^* W_j=I$ and 
 \[
   \varphi(T)=\sum W_j^* T W_j.
 \]  
Let $W$ denote the column matrix with entries $W_i$. With this notation, $\varphi(T)=W^* (I_M\otimes T)W$. In particular, $W$ is an isometry,
since $I=\varphi(I)=W^*W$. Moreover, for
polynomials $p,$
\[
 p(Y)=\varphi(p(X))= W^* (I_M\otimes p(X))W
\]
 and the proof of the reverse direction is complete.

To prove the converse, suppose there is a $N$ and an isometry $V$ such that for all free scalar polynomials $p$,   
\[
  p(Y)= V^* p(I_N\otimes X) V = V^* [I_N\otimes p(X)] V.
\]
 Thus for all matrix free polynomials $\delta$, say of size $d\times d$ (without loss of generality $\delta$ can be assumed square),
\[
 \delta(Y) =  [V \otimes I_d]^* [I_N\otimes \delta(X)]  [V\otimes I_d].
\]
 It follows that $\|\delta(Y)\|\le \|\delta(X)\|.$ 
\end{proof}


\begin{proof}[Proof of Theorem~\ref{thm:freecompact}]
 If for each $X\in E$ there is a $Y\in E$ that does not polynomially dilate to an ampliation of $X$, then, by Lemma \ref{lem:dilation},  for each $X\in E$ there is a $Y\in E$ and a  matrix-valued polynomial $\delta_X$ such that $\|\delta_X(X)\|< \|\delta_X(Y)\|$. An application of Proposition \ref{prop:igor} shows $E$ is not compact.

To prove the converse, suppose there exists $X \in E$ such that every $Y \in E$ polynomially dilates to an ampliation of $X$. Let $\mathcal G$ be an open cover of $E$. There is a $G \in \mathcal G$ and a matrix valued free polynomial $\delta$ such that $X \in G_\delta \subset G$. Since $Y$ polynomially dilates to an ampliation of $X$, it follows that $\|\delta(Y)\|\le \|\delta(X)\|<1$. Hence $Y\in G_\delta \subset G$ and therefore $E\subset G.$
\end{proof}

\end{document}